\documentclass{amsart}
\usepackage{amsthm}
\usepackage{amssymb}
\usepackage{graphicx}
\usepackage{enumerate}

\theoremstyle{plain}
\newtheorem{thm}{Theorem}
\newtheorem{prop}[thm]{Proposition}
\newtheorem*{thm*}{Theorem}
\newtheorem{lem}[thm]{Lemma}

\newcommand{\negative}[1]{#1'}
\newcommand{\constr}[1]{\hat{E}_1(#1)}
\newcommand{\gfextn}[0]{pdf}
\newcommand{\rheinwidth}[0]{1.11in}

\numberwithin{thm}{section}
\numberwithin{equation}{section}


\begin{document}

\title{Sporadic Reinhardt polygons}
\date{\today}

\author{Kevin G. Hare}
\address{Department of Pure Mathematics, University of Waterloo, Waterloo, Ontario, Canada N2L 3G1.}
\email{kghare@uwaterloo.ca}
\thanks{Research of K.G. Hare was partially supported by NSERC}

\author{Michael J. Mossinghoff}
\address{Department of Mathematics, Davidson College, Davidson, NC 28036 USA.}
\email{mimossinghoff@davidson.edu}
\thanks{Research of M.~J. Mossinghoff was partially supported by a grant from the Simons Foundation (\#210069).}

\begin{abstract}
Let $n$ be a positive integer, not a power of two.
A \textit{Reinhardt polygon} is a convex $n$-gon that is optimal in three different geometric optimization problems: it has maximal perimeter relative to its diameter, maximal width relative to its diameter, and maximal width relative to its perimeter.
For almost all $n$, there are many Reinhardt polygons with $n$ sides, and many of them exhibit a particular periodic structure.
While these periodic polygons are well understood, for certain values of $n$, additional Reinhardt polygons exist that do not possess this structured form.
We call these polygons \textit{sporadic}.
We completely characterize the integers $n$ for which sporadic Reinhardt polygons exist, showing that these polygons occur precisely when $n=pqr$ with $p$ and $q$ distinct odd primes and $r\geq2$.
We also prove that a positive proportion of the Reinhardt polygons with $n$ sides are sporadic for almost all integers $n$, and we investigate the precise number of sporadic Reinhardt polygons that are produced  for several values of $n$ by a construction that we introduce.
\end{abstract}

\maketitle

\section{Introduction}\label{secIntroduction}

For a convex polygon in the plane, its \textit{diameter} is the maximum distance between two of its vertices; its \textit{width} is the minimal distance between a pair of parallel lines that enclose it.
A number of natural problems for polygons arise by fixing the number of sides $n$, and fixing one of the four quantities diameter, width, perimeter, or area, and then maximizing or minimizing another one of these attributes.
Six different nontrivial optimization problems for polygons arise in this way, including for example the well-known isoperimetric problem, where the perimeter of a convex $n$-gon is fixed, and one wishes to maximize the area.
In that case, the regular $n$-gon is the unique optimal solution for all $n$, but this is not the case in the other five nontrivial extremal problems in this family.

Prior research has shown that a particular family of polygons is optimal in three of these geometric optimization problems, provided that $n$ is not a power of $2$:
\begin{enumerate}[\quad1.]
\item\label{probA} The isodiametric problem for the perimeter (maximize the perimeter, for a fixed diameter).
\item\label{probB} The isodiametric problem for the width (maximize the width, for a fixed diameter).
\item\label{probC} The isoperimetric problem for the width (maximize the width, for a fixed perimeter).
\end{enumerate}
Problem~\ref{probA} was first studied by Reinhardt in 1922 \cite{Reinhardt22}, and later by others \cite{Datta97,LarmanTamvakis84,Mossinghoff06AMM,Mossinghoff06DCG,Vincze50}.
Problem~\ref{probB} was investigated by Bezdek and Fodor \cite{BezdekFodor00} in 2000, and problem~\ref{probC} was considered by Audet, Hansen, and Messine \cite{AudetHansonMessine09} in 2009.

Before describing the family of polygons that is optimal in these three problems when $n\neq2^m$, we recall that a \textit{Reuleaux polygon} is a closed, convex region in the plane whose boundary consists of a finite number of circular arcs, each with the same curvature, with the property that every pair of parallel lines that sandwiches the region is the same distance apart (that is, Reuleaux polygons have \textit{constant width}).
We say a point on the boundary of a Reuleaux polygon $R$ is a \textit{vertex} of $R$ if it lies at the intersection of two adjacent circular arcs on its boundary.
We briefly recall a number of facts concerning Reuleaux polygons (see \cite{Mossinghoff06AMM} for more details).
First, every Reuleaux polygon $R$ has an odd number of vertices, and each vertex of $R$ is equidistant from all of the points on one of the circular arcs that form the boundary of $R$.
Second, connecting all pairs of vertices at maximal distance from one another in a Reuleaux polygon forms a \textit{star polygon}---a closed path in the plane consisting of an odd number of line segments, each of which intersects all of the others.
The sum of the measures of the angles at the vertices of a star polygon is $\pi$, and each line segment comprising the star polygon associated with a Reuleaux polygon has the same length, equal to the diameter of the Reuleaux polygon.
Third, we can recover the Reuleaux polygon $R$ from its associated star polygon $S$ by visiting each vertex $v$ of $S$ and drawing a circular arc between the two vertices adjacent to $v$ in $S$, with radius equal to the length of each line segment in $S$.
Last, an ordinary polygon can always be inscribed in a Reuleaux polygon with the same diameter.

A polygon with $n\neq2^m$ sides that is optimal in the three problems described above is called a \textit{Reinhardt polygon}, which we define as as an equilateral convex polygon $P$ that can be inscribed in a Reuleaux polygon $R$ having the same diameter as $P$ in such a way that every vertex of $R$ is a vertex of $P$ \cite{AudetHansonMessine09,BezdekFodor00,Reinhardt22}.
If $n$ is odd, then the regular $n$-gon is a Reinhardt polygon, but this is not the case when $n$ is even.
In addition, for almost all $n\geq3$ there is more than one Reinhardt polygon with $n$ sides---this is the case for all $n$ except those of the form $p$, $2p$, or $2^m$, where $p$ is prime \cite{Mossinghoff11}.
For example, Figure~\ref{figAll21} exhibits the ten different Reinhardt polygons having $n=21$ sides.
These polygons are all distinct, in that one cannot be obtained from another by some combination of rotations and flips.

\begin{figure}[t]
\caption{Reinhardt polygons with $n=21$ sides.}\label{figAll21}
\begin{tabular}{c@{\quad}c@{\quad}c@{\quad}c}\\[0pt]
\includegraphics[width=\rheinwidth]{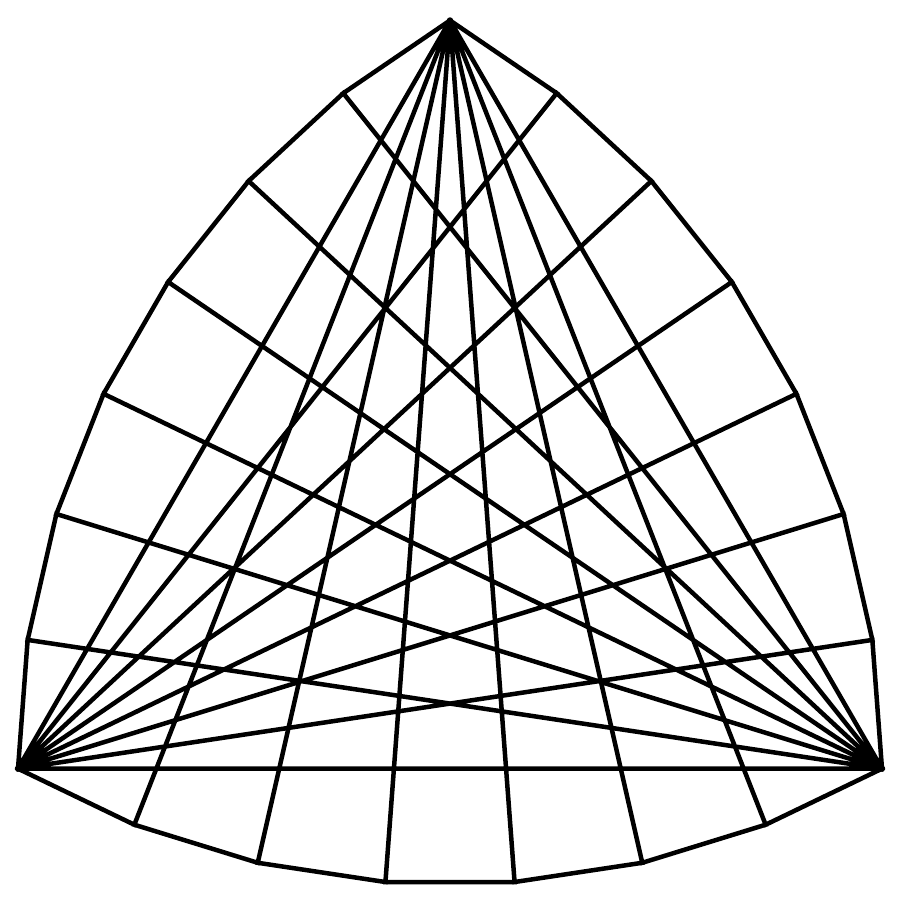} &
\includegraphics[width=\rheinwidth]{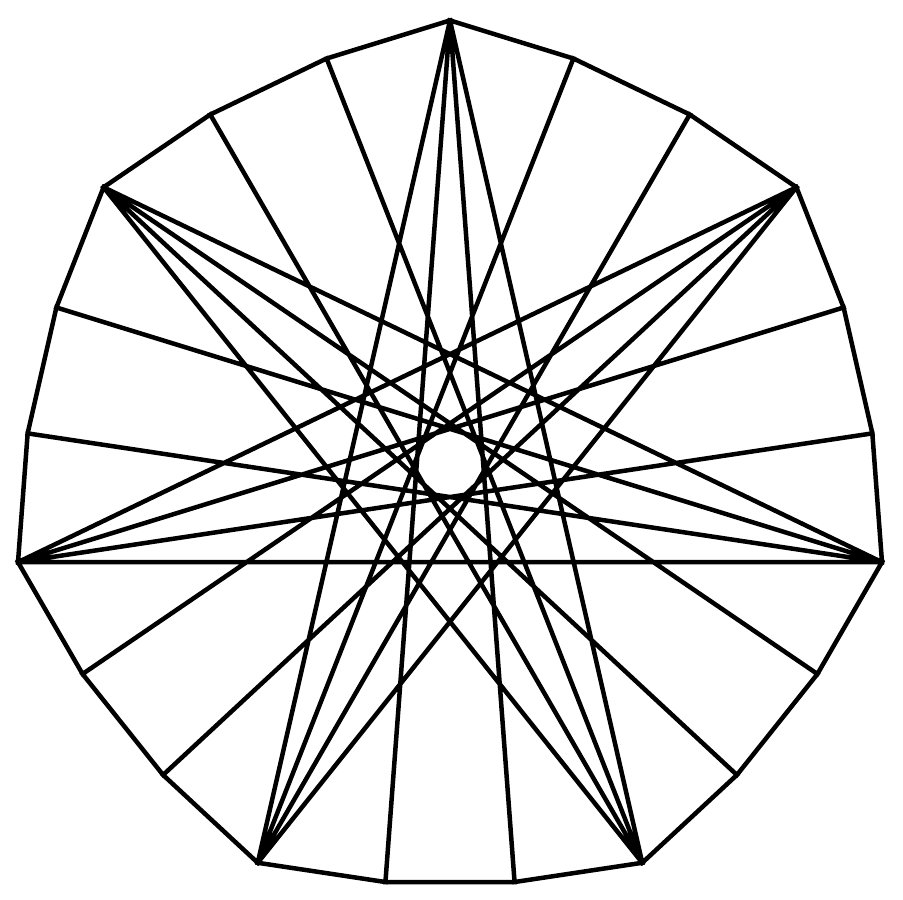} &
\includegraphics[width=\rheinwidth]{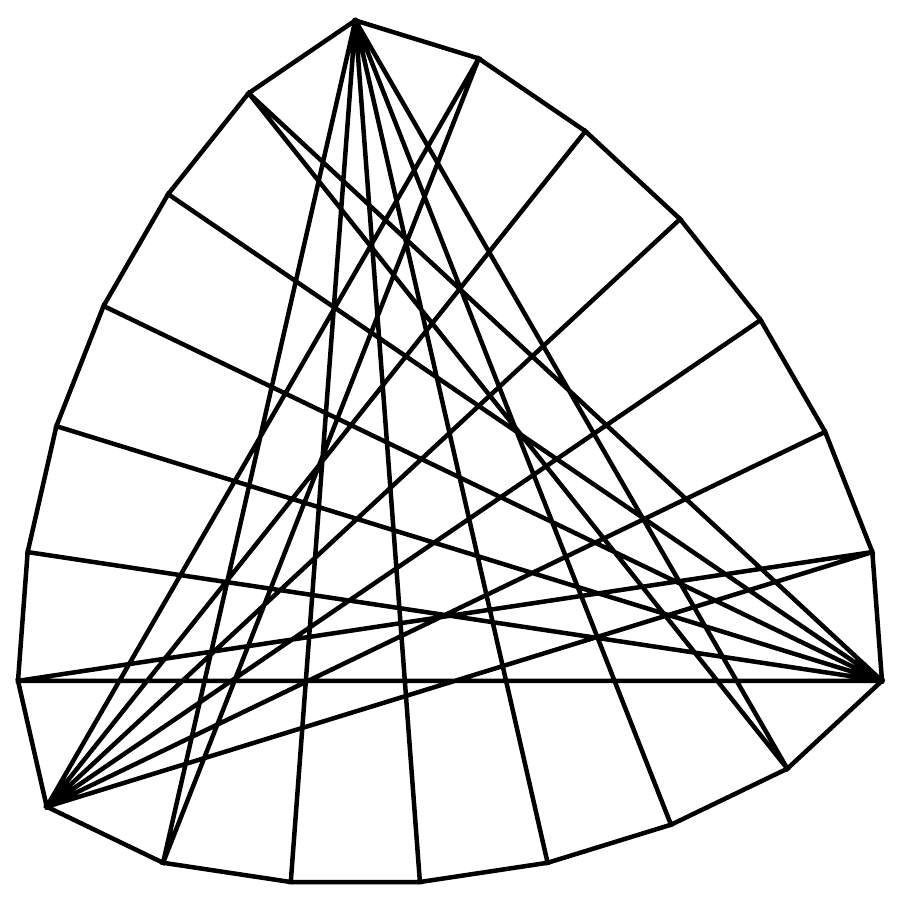} &
\includegraphics[width=\rheinwidth]{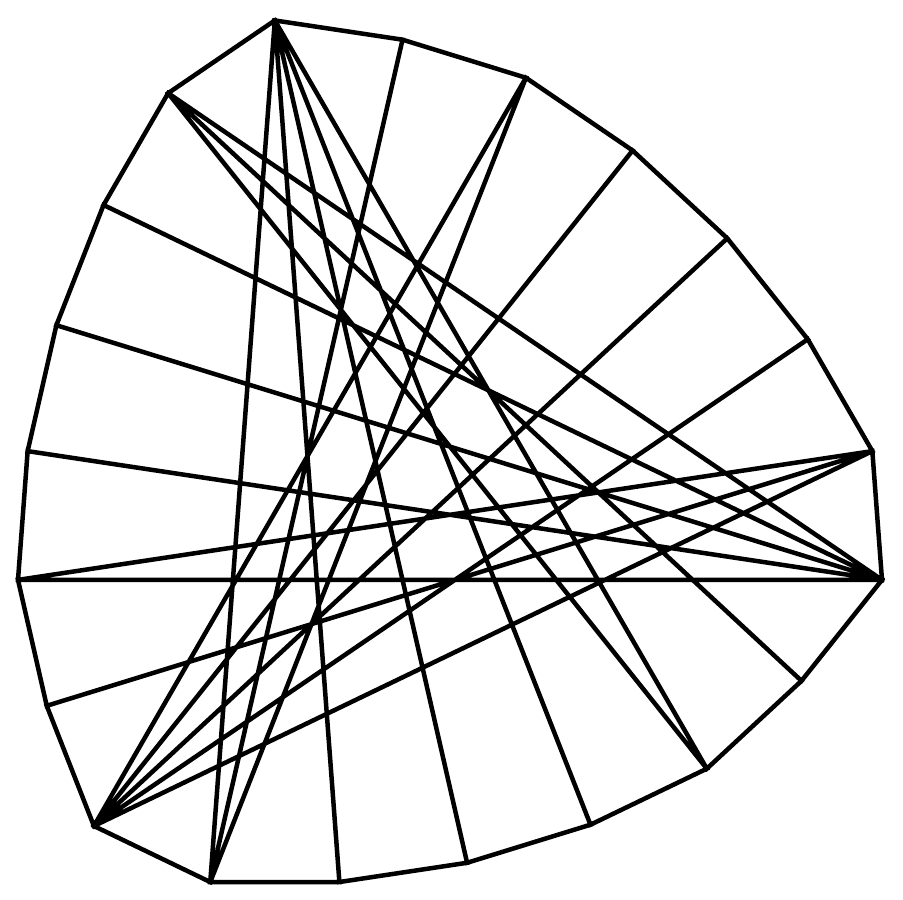}\\
\small (a) $[(7)^3]$ &
\small (b) $[(3)^7]$ &
\small (c) $[(5,1,1)^3]$ &
\small (d) $[(4,2,1)^3]$\\[12pt]
\includegraphics[width=\rheinwidth]{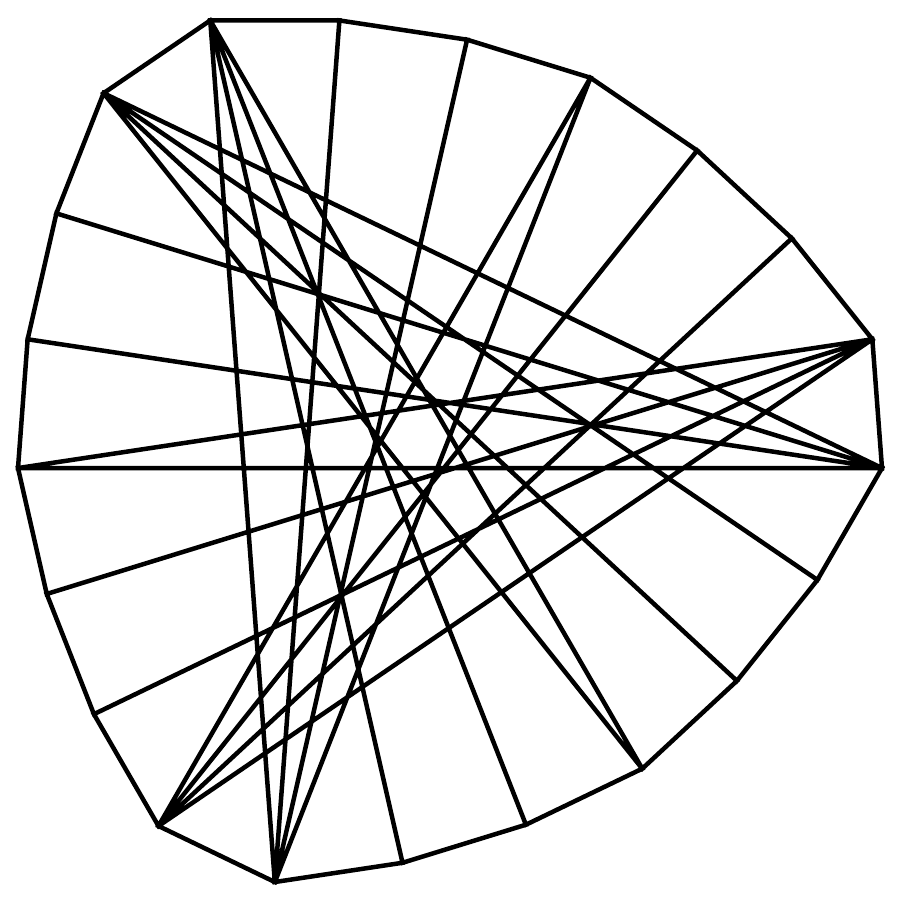} &
\includegraphics[width=\rheinwidth]{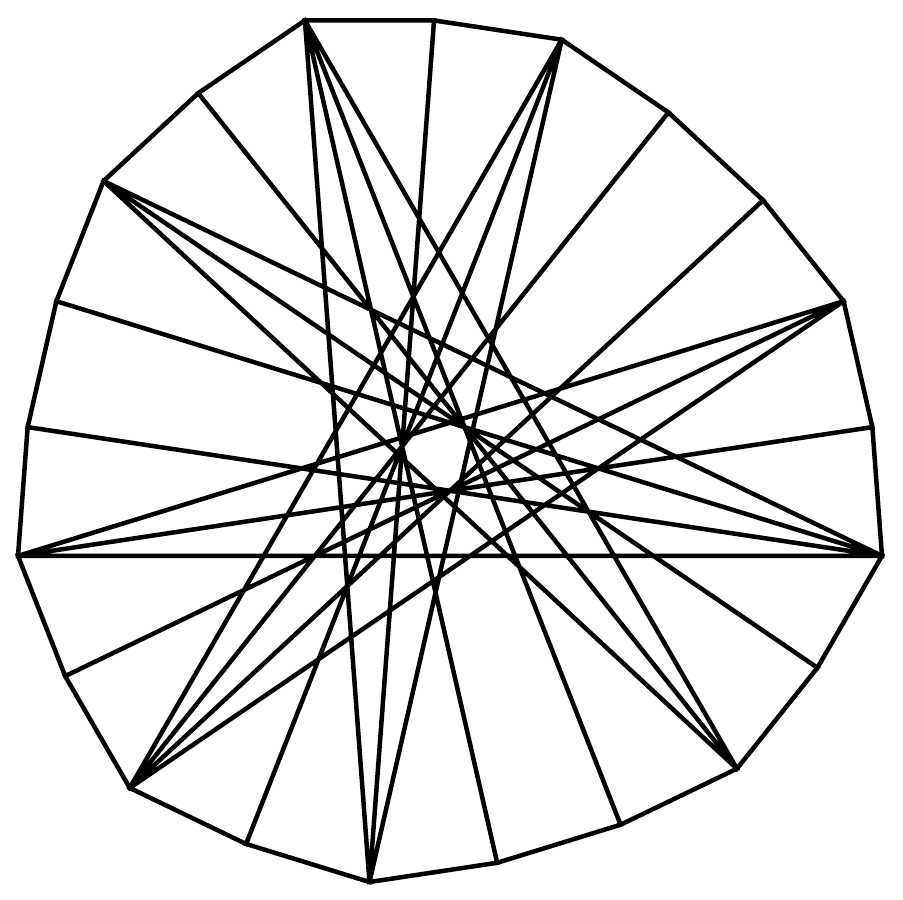} &
\includegraphics[width=\rheinwidth]{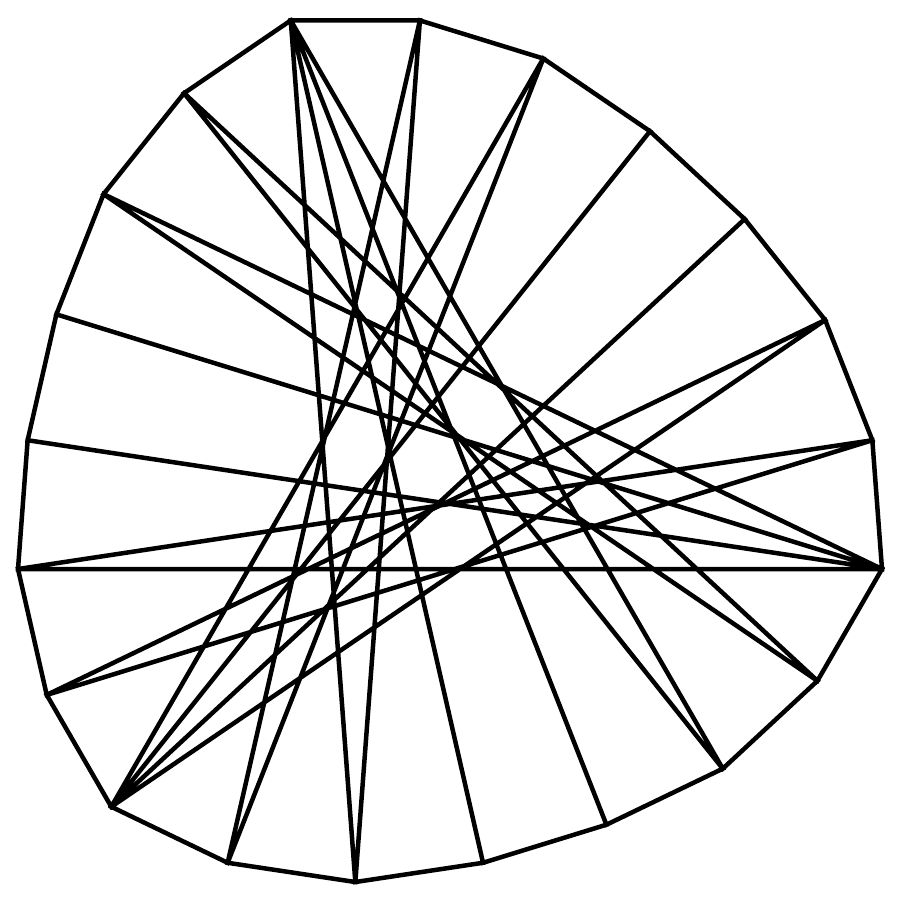} &
\includegraphics[width=\rheinwidth]{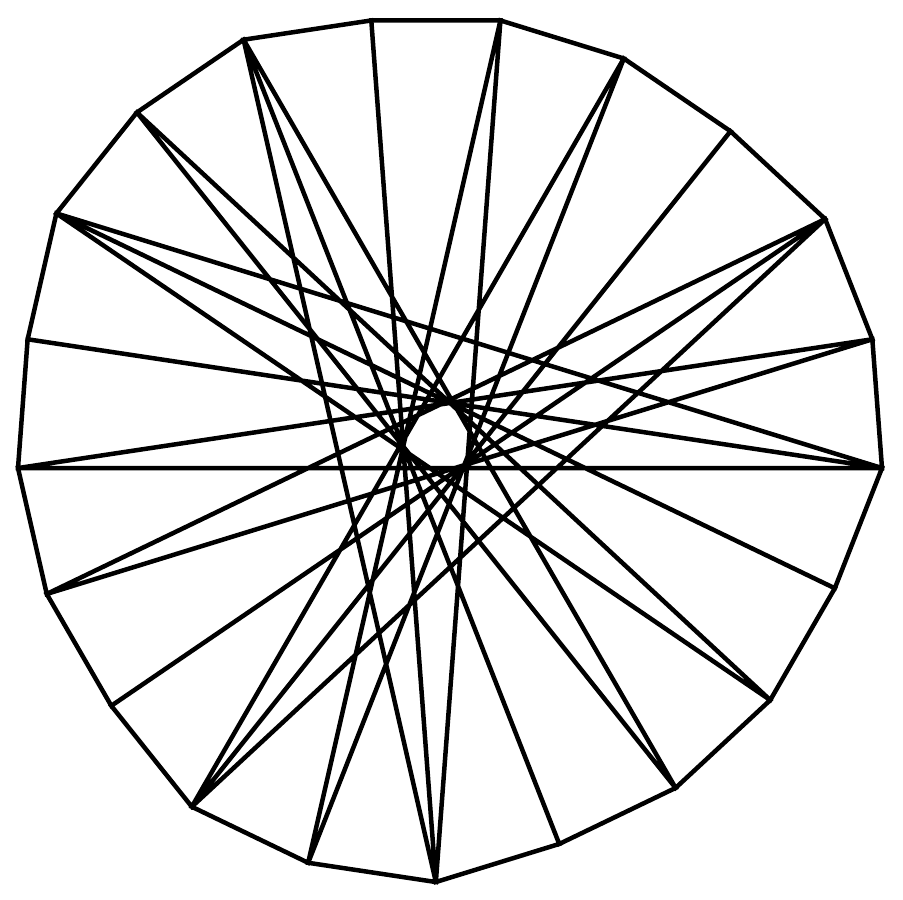}\\
\small (e) $[(3,3,1)^3]$ &
\small (f) $[(3,2,2)^3]$ &
\small (g) $[(3,1,1,1,1)^3]$ &
\small (h) $[(2,2,1,1,1)^3]$\\[12pt]
&
\includegraphics[width=\rheinwidth]{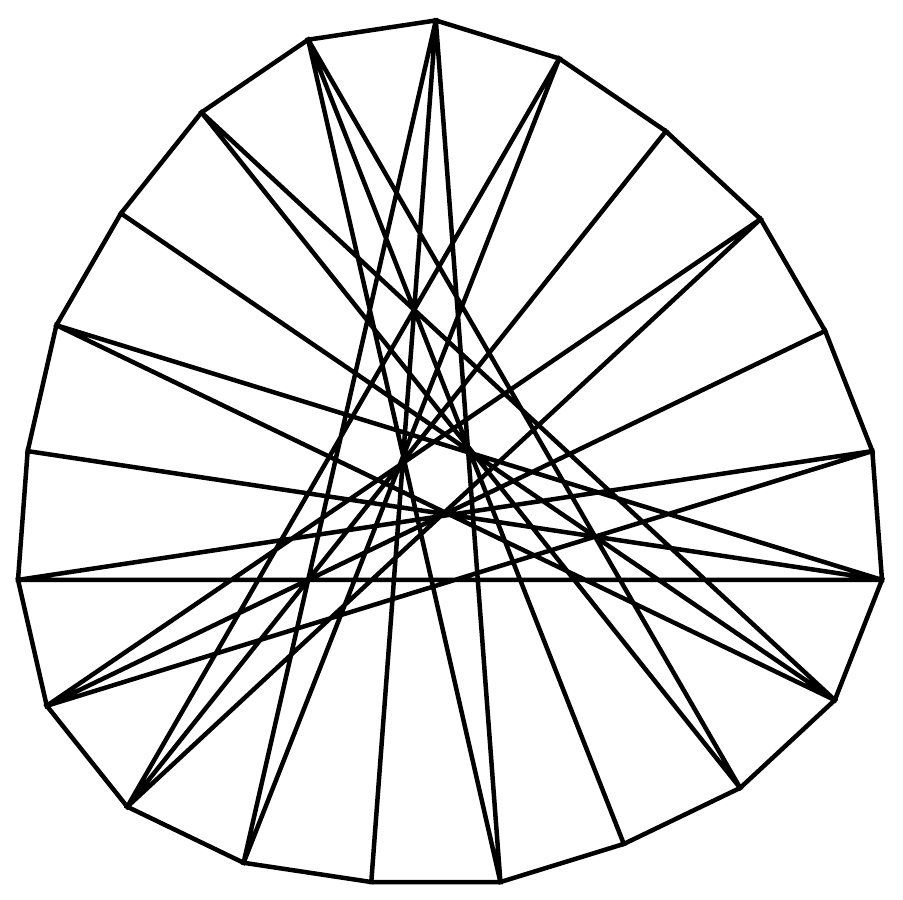} &
\includegraphics[width=\rheinwidth]{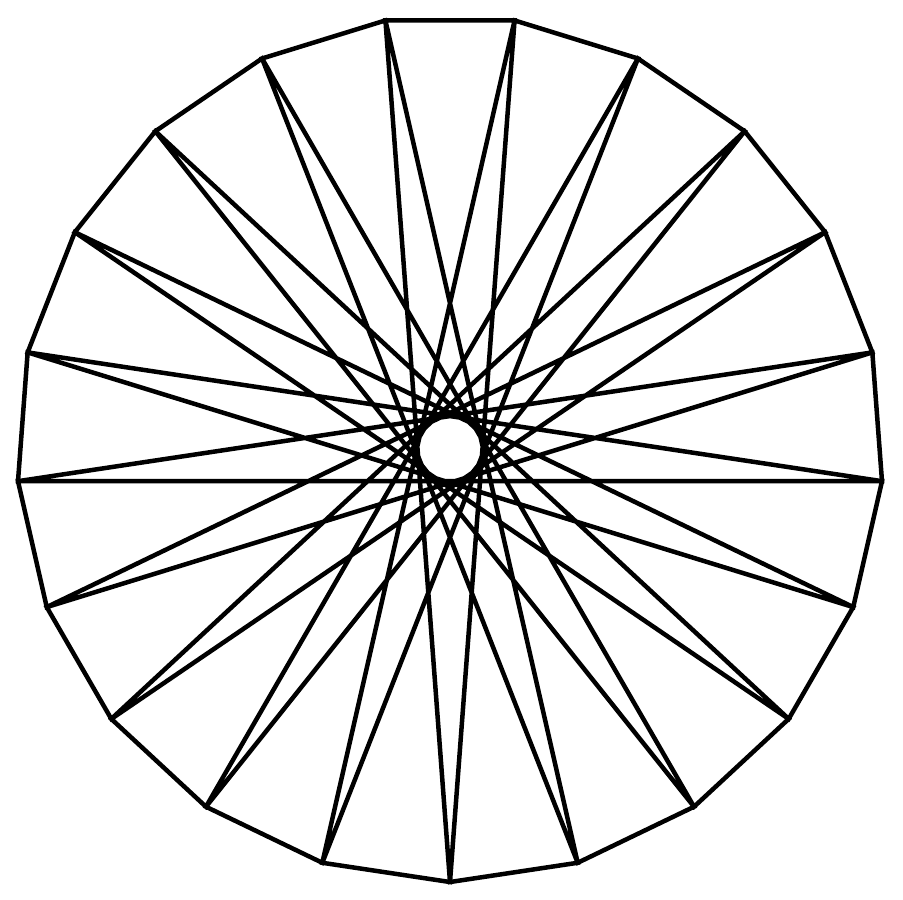}\\
&
\small (i) $[(2,1,2,1,1)^3]$ &
\small (j) $[(1)^{21}]$
\end{tabular}
\end{figure}

We can describe a Reinhardt polygon $P$ in a compact way by focusing on the star polygon $S$ associated with the Reuleaux polygon $R$ that circumscribes $P$.
Suppose that $P$ has $n$ vertices.
Vertices of $P$ that are not vertices of $R$ lie on the boundary of $R$, and since $P$ is equilateral, these vertices subdivide the circular arcs on the boundary of $R$ into subarcs of equal length.
The angle at a vertex $v$ of $S$ then has measure $k\pi/n$, where $k$ is the number of sides of $P$ that are inscribed in the arc of $R$ that lies opposite $v$.
We can thus describe a Reinhardt polygon by naming the sequence of these integers $k$ as one circumnavigates the star polygon $S$.
If $S$ has $r$ vertices, then we denote this sequence by $[k_1,k_2,\ldots,k_r]$.
For example, the sequence for the $21$-sided polygon in Figure~\ref{figAll21}(a), where the underlying star polygon is an equilateral triangle, each of whose angles is subdivided into sevenths, is $[7,7,7]$, and we abbreviate this by $[(7)^3]$.
In the same way, the regular henicosagon of Figure~\ref{figAll21}(j) is denoted by a sequence of $21$ ones: $[(1)^{21}]$.
A Reinhardt polygon with $n$ sides is therefore denoted by a particular composition of $n$ into an odd number of parts.
Naturally, we consider two such compositions to be equivalent if one can be obtained from the other by a combination of cyclic shifts and list reversals (corresponding to rotations and flips of the polygon), and so we consider equivalence classes of such compositions under a dihedral action.
We call such an equivalence class a \textit{dihedral composition}.

Not every dihedral composition of an integer $n$ into an odd number of parts corresponds to a Reinhardt polygon; an extra condition is required to ensure that the path determined by a list of integers is closed.
Reinhardt \cite{Mossinghoff06AMM,Reinhardt22} obtained a characterization for valid compositions in terms of an associated polynomial: given $[k_1,\ldots,k_r]$ with $r$ odd and $\sum_{i=1}^r k_i = n$, form the polynomial
\[
F(z) = 1 - z^{k_1} + z^{k_1+k_2} - \cdots + z^{k_1+\cdots+k_{r-1}}.
\]
Then $[k_1,\ldots,k_r]$ corresponds to a Reinhardt polygon if and only if $\Phi_{2n}(z) \mid F(z)$, where $\Phi_m(z)$ denotes the $m$th cyclotomic polynomial.
We say $F(z)$ is a \textit{Reinhardt polynomial} for $n$ if $F(0)=1$, $\deg(F)<n$, $F$ has an odd number of terms, the nonzero coefficients of $F$ alternate $\pm1$, and $\Phi_{2n}(z)\mid F(z)$.
For example, the polynomials associated with the polygons of Figures~\ref{figAll21}(a) and~\ref{figAll21}(b) are respectively $1-z^7+z^{14}$ and $1-z^3+z^6-z^9+z^{12}-z^{15}+z^{18}$.

Let $E(n)$ denote the number of Reinhardt polygons with $n$ sides, counting two polygons as distinct only if one cannot be obtained from the other by a combination of rotations and flips.
One may determine $E(n)$ by enumerating the Reinhardt polynomials for $n$; this strategy shows for example that Figure~\ref{figAll21} exhibits the complete set of $21$-sided Reinhardt polygons.
Such computations reveal that for many values of $n$, every Reinhardt polygon with $n$ sides exhibits special structure, in that the corresponding dihedral composition is periodic.
We call such a polygon a \textit{periodic} Reinhardt polygon.
From Figure~\ref{figAll21}, we see that every henicosagonal Reinhardt polygon is periodic, since each corresponding dihedral composition in this figure has the form $[(k_1,\ldots,k_s)^d]$ for some divisor $d$ of $21$.
In fact, every Reinhardt polygon with $n<30$ sides is periodic, but at $n=30$, there are $38$ periodic Reinhardt polygons, plus three that do not exhibit such structure.
We call these polygons \textit{sporadic}, and the three for $n=30$ are exhibited in Figure~\ref{figSporadic30}.
They also appear at $n=42$ and at $n=45$, but at no other integers $n<60$.

\begin{figure}[t]
\caption{Sporadic Reinhardt polygons for $n=30$.}\label{figSporadic30}
\begin{tabular}{c@{\qquad}c}\\[0pt]
\includegraphics[width=1.5in]{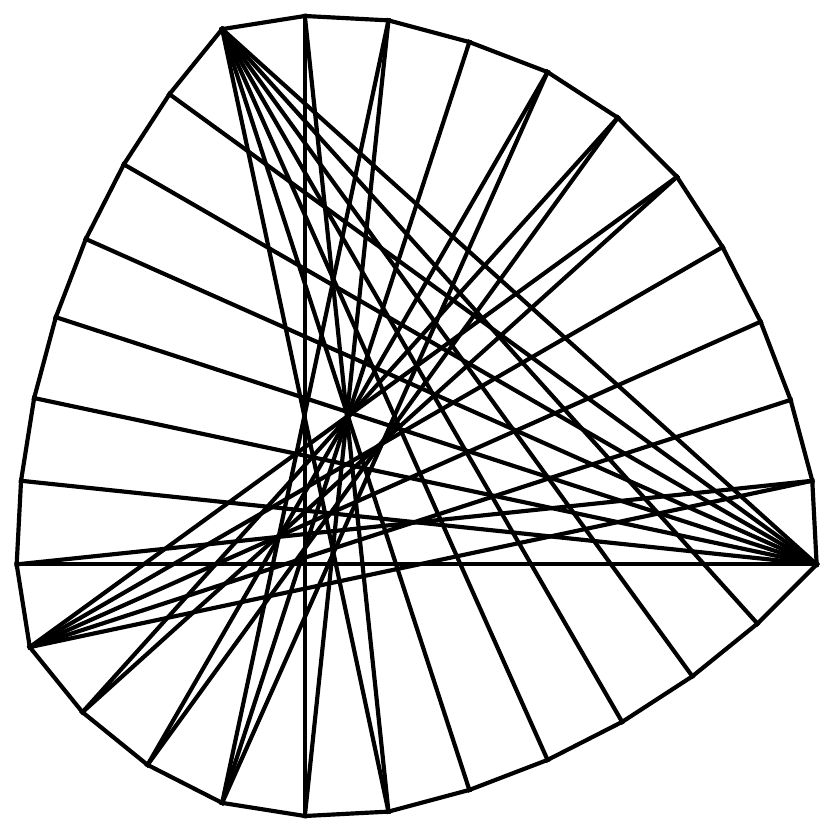} &
\includegraphics[width=1.5in]{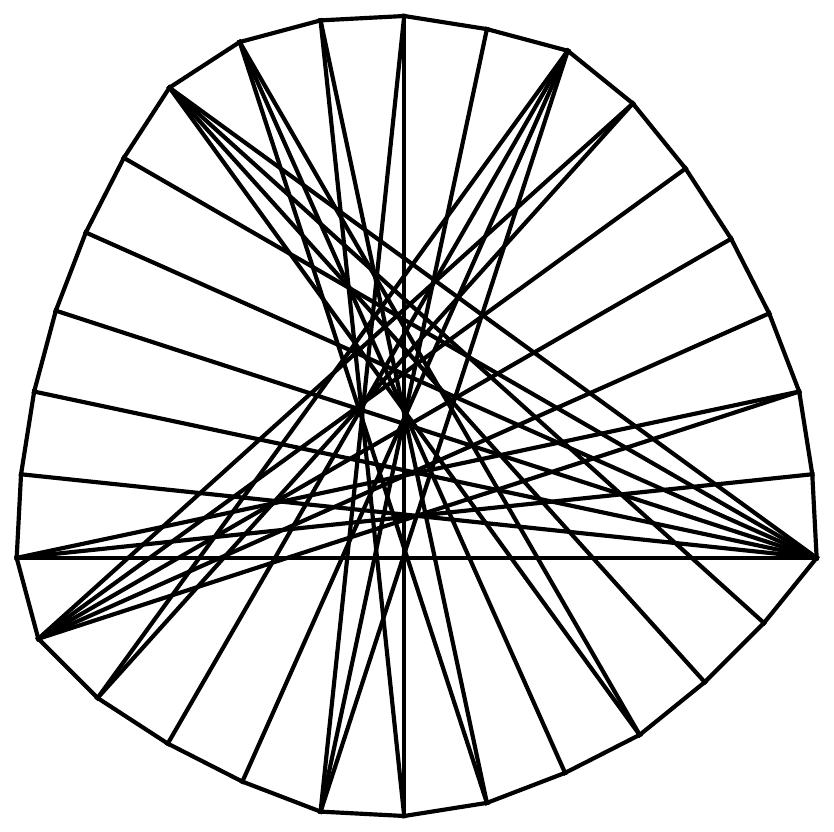}\\
(a) \small $[7,6,1,1,1,1,2,1,1,1,1,1,4,1,1]$ &
(b) \small $[6,3,1,2,1,1,1,1,2,3,1,1,4,1,2]$\\[12pt]
\multicolumn{2}{c}{\includegraphics[width=1.5in]{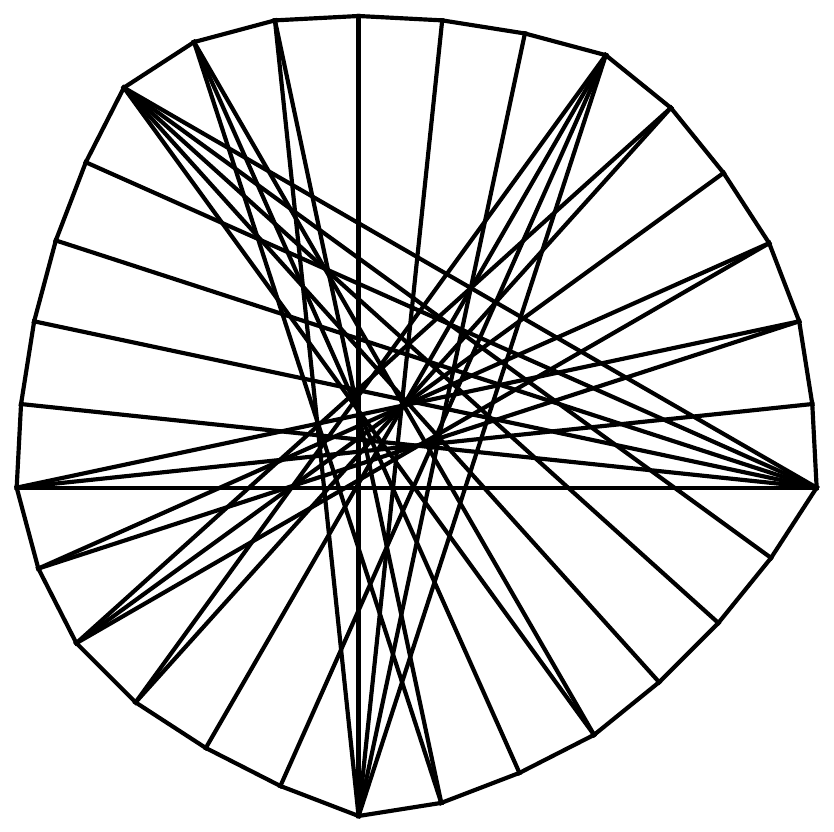}}\\
\multicolumn{2}{c}{(c) \small $[5,4,1,2,1,1,4,3,1,1,2,1,1,1,2]$}
\end{tabular}
\end{figure}

Let $E_0(n)$ denote the number of periodic Reinhardt polygons having $n$ sides, and let $E_1(n)$ denote the number of sporadic Reinhardt polygons with $n$ sides (again, with both counts taken under dihedral equivalence), so $E(n) = E_0(n)+E_1(n)$.
In \cite{Mossinghoff11}, the second author obtained a number of results regarding these quantities, including an exact value for $E_0(n)$:
\begin{equation}\label{eqnPeriodicFormula}
E_0(n) = \sum_{\substack{d\mid n\\d>1}} \mu(2d) D(n/d),
\end{equation}
where $\mu(\cdot)$ is the M\"obius function,
\[
D(m) = 2^{\lfloor(m-3)/2\rfloor} + \frac{1}{4m}\sum_{\substack{d\mid m\\2\nmid d}}2^{m/d}\varphi(d),
\]
and $\varphi(\cdot)$ is Euler's totient function.
This formula follows from some combinatorial analysis, after observing that every composition of an integer $n\geq3$ of the form $[(k_1,\ldots,k_s)^d]$ with $s$ and $d$ odd corresponds to a Reinhardt polygon.
Gashkov \cite{Gashkov07} obtained a similar result, but assuming cyclic equivalence classes instead of dihedral ones.

In \cite{Mossinghoff11}, it was also shown that $E_1(n)=0$ if $n$ has the form $n=2^a p^{b+1}$, for some odd prime $p$, where $a$ and $b$ are nonnegative integers, and that $E(n)=1$ if and only if $n=p$ or $n=2p$ for some odd prime $p$.
In addition, some computations in \cite{Mossinghoff11} indicated that $E_1(n)=0$ for $40$ different values of $n$ of the form $n=pq$, with $p$ and $q$ distinct odd primes, and that $E_1(n)>0$ for many integers having neither the form $n=pq$ nor $n=2^a p^{b+1}$.
Two problems were posed in that article: first, determine if $E_1(pq)=0$ whenever $p$ and $q$ are distinct odd primes; second, determine if $E_1(n)>0$ whenever $n=pqr$, with $p$ and $q$ distinct odd primes and $r\geq2$.
In this article, we prove both of these assertions.

\begin{thm}\label{thmPeriodic}
If $n=pq$, with $p$ and $q$ distinct odd primes, then every Reinhardt polygon with $n$ sides is periodic, that is, $E_1(n)=0$.
\end{thm}

Using Theorem~\ref{thmPeriodic} and \eqref{eqnPeriodicFormula}, we therefore obtain an exact formula for the number of Reinhardt polygons with $n$ sides (under dihedral equivalence) when $n$ is a product of two distinct odd primes:
\begin{equation}\label{eqnFormulaPQ}
E(pq) = 2^{(p-3)/2} + \frac{2^{p-1}+p-1}{2p}
+ 2^{(q-3)/2} + \frac{2^{q-1}+q-1}{2q}
-1.
\end{equation}
Section~\ref{sec:prob1} describes the proof of Theorem~\ref{thmPeriodic}.
It relies on the structure of a principal ideal in $\mathbb{Z}[z]$ generated by a cyclotomic polynomial.

\begin{thm}\label{thmSporadic}
If $n=pqr$, with $p$ and $q$ distinct odd primes and $r\geq2$, then there exists a sporadic Reinhardt polygon with $n$ sides, that is, $E_1(n)>0$.
\end{thm}

Section~\ref{secBuildingPolygons} contains the proof of Theorem~\ref{thmSporadic}.
This proof is constructive, describing a method for creating a sizable family of sporadic Reinhardt polygons for any qualifying integer $n$.

In \cite{Mossinghoff11}, the exact value of $E_1(n)$ was computed for $24$ different integers $n$ having the form specified in Theorem~\ref{thmSporadic}, and in each of these cases it was found that $E_0(n) > E_1(n)$, and often $E_0(n)$ was in fact several orders of magnitude larger than $E_1(n)$.
Despite this, it was conjectured in \cite{Mossinghoff11} that $E_1(n)>E_0(n)$ for almost all positive integers.
We obtain some information on the size of $E_1(n)$ in this article, showing that, in a particular sense, a positive proportion of the Reinhardt polygons with $n$ sides are sporadic, for almost all $n$.
The following result is also established in Section~\ref{secBuildingPolygons}.

\begin{thm}\label{thmNumber}
Suppose $p$ and $q$ are fixed odd primes with $p<q$, and let $\epsilon>0$.
Then for all sufficiently large integers $r$ having no odd prime divisor less than $p$, we have
\[
\frac{E_1(pqr)}{E(pqr)} > \frac{2^p - 2}{p2^q + 2^p - 2} - \epsilon.
\]
\end{thm}

Section~\ref{secConstructions} explores the precise number of sporadic Reinhardt polygons produced by our method, and compares this number with the exact value of $E_1(n)$, for the $24$ different integers $n$ where this is known.
We find for example that our method constructs all of the sporadic polygons for some of these values.
This section also treats the case $n=105$ in some detail, and here our method produces far more Reinhardt polygons than what is suggested by the bound of Theorem~\ref{thmNumber}.
It was posited in \cite{Mossinghoff11} that $n=105$ is the smallest integer where $E_1(n)$ exceeds $E_0(n)$.
Our calculations here provide some further empirical evidence for this assertion.

\section{Proof of Theorem~\ref{thmPeriodic}}\label{sec:prob1}

A theorem of de Bruijn \cite{deBruijn53} (see also \cite{Redei54,Schoenberg64}) states that the principal ideal generated by the cyclotomic polynomial $\Phi_m(z)$ in $\mathbb{Z}[z]$ is generated by the collection of polynomials $\{\Phi_p(z^{m/p}) : \textrm{$p$ is prime and $p\mid m$}\}$.
Note that each term in this generating set is certainly a multiple of $\Phi_m(z)$, since
\[
\Phi_p(z^{m/p}) = \prod_{\substack{d\mid m\\d\nmid\frac{m}{p}}}\Phi_d(z).
\]
Thus, if $\{p_1, \ldots, p_r\}$ are the odd prime divisors of a (possibly even) positive integer $n$, and if $\Phi_{2n}(z)\mid F(z)$, then there exist integer polynomials $f_0(z)$, \ldots, $f_r(z)$ such that
\[
F(z) = f_0(z)(z^n+1) + \sum_{i=1}^r f_i(z) \Phi_{p_i}(z^{2n/p_i}).
\]
However, since
\[
\Phi_{p_i}(z^{2n/p_i}) - z^{n/p_i}(z^n+1)\frac{z^{(p_i-1)n/p_i}-1}{z^{2n/p_i}-1} = \Phi_{p_i}(-z^{n/p_i}),
\]
then an equivalent condition for $\Phi_{2n}(z)\mid F(z)$ is the existence of polynomials $f_i(z)$ such that
\begin{equation}\label{eqnCycGen}
F(z) = f_0(z)(z^n+1) + \sum_{i=1}^r f_i(z) \Phi_{p_i} (-z^{n/p_i}).
\end{equation}
Suppose $F(z)$ is the Reinhardt polynomial corresponding to a periodic Reinhardt polygon with $n$ sides, which arises from a composition of $n$ of the form $[(k_1,\ldots,k_s)^d]$, where $d$ and $s$ are odd and $d\geq3$.
Let $m=\sum_{i=1}^s k_i$.
Then
\[
F(z) = f(z)\sum_{i=0}^{d-1} (-1)^i z^{mi},
\]
where $f(0)=1$, $\deg(f)<m$, and the $s$ nonzero coefficients of $f$ alternate $\pm1$.
Select $j$ so that the odd prime divisor $p_j$ of $n$ divides $d$, and let $e=d/p_j$.
Then
\[
F(z) = f(z)\biggl(\sum_{i=0}^{e-1}(-1)^i z^{mi}\biggr)\Phi_{p_j}(-z^{n/p_j}),
\]
and this has the form of \eqref{eqnCycGen} if one takes $f_j(z)=f(z)\sum_{i=0}^{e-1}(-1)^i z^{mi}$, and every other $f_i(z)=0$.
Conversely, if $F(z)$ is a polynomial formed by using \eqref{eqnCycGen} with each $f_i(z)=0$ except for one with positive index $j$, and taking this polynomial $f_j(z)$ to have alternating $\pm1$ nonzero coefficients, an odd number of terms, $f_j(0)=1$, and $\deg(f_j)<n/p_j$, then clearly $F(z)$ corresponds to a periodic Reinhardt polygon for $n$.
Thus, the polygons corresponding to the Reinhardt polynomials in \eqref{eqnCycGen} in which each $f_i(z)=0$ except for one with positive index are precisely the periodic Reinhardt polygons for $n$.

In order to establish Theorem~\ref{thmPeriodic}, we therefore need only prove that every Reinhardt polynomial for $n=pq$, where $p$ and $q$ are distinct odd primes, can be represented by using just the $i=1$ or the $i=2$ term of \eqref{eqnCycGen}.
We may state this requirement in a compact way as a divisibility condition on $F(z)$.
Theorem~\ref{thmPeriodic} then follows immediately from the following proposition.

\begin{prop}\label{prop:pqRep}
Let $n = p q$ with $p$ and $q$ distinct odd primes. 
If $F(z)$ is a Reinhardt polynomial for $n$, then either $\Phi_p(-z^q)\mid F(z)$ or $\Phi_q(-z^p)\mid F(z)$.
\end{prop}

In order to establish this statement, we require two preliminary results.

\begin{lem}\label{lemDecomp}
Let $n = p q$ with $p$ and $q$ distinct odd primes, and suppose that $F(z)$ is a Reinhardt polynomial
for $n$.
Then there exist polynomials $f_1(z)$ and $f_2(z)$ with integer coefficients, $\deg(f_1) < q$, $\deg(f_2) < p$, and
\begin{equation}\label{eqnDecomp}
F(z) = f_1(z) \Phi_p(-z^q) + f_2(z) \Phi_q(-z^p).
\end{equation}
Further, we may choose $f_1(z)$ and $f_2(z)$ to have all their coefficients in $\{-1,0,1\}$.
\end{lem}

\begin{proof}
First, as in \cite{Schoenberg64}, we note that if $s$ and $t$ are positive integers and $s=tu+v$, with $u$ and $v$ integers and $0\leq v<t$, then the division algorithm in $\mathbb{Z}[z]$ certainly produces
\[
\sum_{k=0}^{s-1} z^k = \left(\sum_{k=1}^u z^{s-kt}\right)\left(\sum_{k=0}^{t-1} z^k\right) + \sum_{k=0}^{v-1} z^k.
\]
It follows that the Euclidean algorithm in $\mathbb{Z}[z]$ applied to the cyclotomic polynomials $\Phi_p(z)$ and $\Phi_q(z)$ produces integer polynomials $a(z)$ and $b(z)$, with $\deg(a)<q$ and $\deg(b)<p$, such that $a(z)\Phi_p(z) + b(z)\Phi_q(z) = 1$.
Since $\Phi_m(-z)=\Phi_{2m}(z)$ when $m$ is odd, we may rewrite this as $a(-z)\Phi_{2p}(z) + b(-z)\Phi_{2q}(z) = 1$.
Write $F(z) = \Phi_{2pq}(z) h(z)$, so that $\deg(h) < p+q-1$.
We therefore find that
\[
F(z) = h(z)a(-z)\Phi_{2 p q}(z)\Phi_{2 p}(z) + h(z)b(-z)\Phi_{2 p q}(z)\Phi_{2 q}(z).
\]
Let $f_1(z)$ and $c(z)$ be the integer polynomials satisfying $\deg(f_1)<q-1$ and
\[
h(z)a(-z) = c(z)\Phi_{2q}(z) + f_1(z),
\]
so that
\[
F(z) = f_1(z)\Phi_{2 p q}(z)\Phi_{2 p}(z) + (h(z)b(-z) + c(z)\Phi_{2p}(z))\Phi_{2 p q}(z)\Phi_{2 q}(z).
\]
We choose $f_2(z) = h(z)b(-z) + c(z)\Phi_{2p}(z)$ and one may check easily that $\deg(f_2)<p$.
Since $\Phi_{pq}(z)\Phi_p(z) = (z^{pq}-1)/(z^q-1) = \Phi_p(z^q)$, we see that $\Phi_{2pq}(z)\Phi_{2p}(z) = \Phi_p(-z^q)$, and similarly when the roles of $p$ and $q$ are reversed.
We have therefore produced integer polynomials $f_1(z)$ and $f_2(z)$ satisfying the required degree bounds and \eqref{eqnDecomp}, and the first statement of the Lemma follows.

For the second statement, let 
\begin{equation}\label{eqnST}
f_1(z)\Phi_p(-z^q) = \sum_{i=0}^{pq-1} s_i z^i
\textrm{\quad and\quad}
f_2(z)\Phi_q(-z^p) = \sum_{j=0}^{pq-1} t_j z^j,
\end{equation}
so that $s_{i+q}=-s_i$ for $0\leq i<q(p-1)$ and $t_{j+p}=-t_j$ for $0\leq j<(q-1)p$, and let $F(z) = \sum_{i=0}^{pq-1}u_i z^i$.
It is straightforward to verify that if $(f_1(z), f_2(z))$ satisfy \eqref{eqnDecomp}, then so does
$(f_1(z) + m\Phi_{2q}(z), f_2(z) - m\Phi_{2p}(z))$ for any integer $m$.
Since $u_0=1$, we may therefore assume that $s_0=1$ and $t_0=0$.
Select integers $p'$ and $q'$ so that $p p' + q q' = 1$, and select an integer $i$ with $1\leq i<pq$.
Let $e$ be the integer with $0 \leq e < pq$ such that $e \equiv i p p'$ mod $pq$.
Since $e \equiv 0$ mod $p$, $e \equiv i$ mod $q$, and $t_0=0$, we see that $u_e = s_e + t_e = \pm s_i$.
As a Reinhardt polynomial, each coefficient of $F(z)$ is $-1$, $0$, or $1$, so consequently $s_i\in\{-1,0,1\}$ for each $i$.

If there exists an integer $i$ with $1\leq i<pq$ and $s_i = 0$, then a similar argument, using $0 \leq e < pq$ such that $e \equiv i p p' + j q q'$ mod $pq$, shows that $t_j \in \{-1,0,1\}$ for each $j$.
Suppose then that each $s_i=\pm1$.
If $s_i+t_i \neq 0$ for all $i$ then $s_i + t_i = (-1)^i$, and we can 
    use the trivial decomposition $s_i = (-1)^i$ and $t_i = 0$ for each $i$.
Hence assume that there exists a $k$ such that $s_k + t_k = 0$. 
In particular, $t_k = -s_k = \mp 1$ is odd.
Replace $f_1(z)$ by $f_1^*(z)=f_1(z)-\Phi_{2q}(z)$ and $f_2(z)$ by $f_2^*(z) = f_2(z)+\Phi_{2p}(z)$, and write $f_1^*(z)\Phi_p(-z^q) = \sum_{i=0}^{pq-1} s^*_i z^i$ and $f_2^*(z)\Phi_q(-z^p) = \sum_{j=0}^{pq-1} t^*_j z^j$.
Thus, each $s^*_i$ is even, $s^*_0=0$, $t^*_0=1$, and $t^*_k$ is even.
If $e$ satisfies $0\leq e<pq$, $e\equiv0$ mod $p$, and $e\equiv k$ mod $q$, then $u_e=\pm s^*_0\pm t^*_k = \pm t^*_k \in \{-1, 0, 1\}$, and consequently, as $t^*_k$ is even, $t^*_k=0$.
In the same way as above, the fact that $s^*_0=0$ implies that each $t^*_j\in\{-1,0,1\}$, and because $t^*_k=0$, we now conclude that each $s^*_i\in\{-1,0,1\}$ as well.
In fact, because each $s_i^*$ is even, we have that $f_1^*(z)=0$ and so $f_1(z)=\Phi_{2q}(z)$.
\end{proof}

We say that polynomials $f_1(z)$ and $f_2(z)$ satisfying all the conditions of Lemma~\ref{lemDecomp} form a \textit{decomposition} of the Reinhardt polynomial $F(z)$, and we say that a decomposition is \textit{trivial} if either $f_1(z)\in\{0,\Phi_{q}(-z)\}$, or $f_2(z)\in\{0,\Phi_{p}(-z)\}$.
Note that we do not require that $f_1(z)=0$ or $f_2(z)=0$ in a trivial decomposition, since other trivial configurations may occur.
For example, with $p=3$ and $q=5$, one may check that choosing $f_1(z)=z(1-z)$ and $f_2(z)=\Phi_3(-z)$ in \eqref{eqnDecomp} produces the same Reinhardt polynomial as the selection $f_1(z)=1-z^3+z^4$ and $f_2(z)=0$.

From the proof of the lemma, we see that both zero and nonzero coefficients appear among both the $s_i$ and the $t_j$ in any nontrivial decomposition.
We require one further property of a nontrivial decomposition.

\begin{lem}\label{lemParity}
Let $n = p q$ with $p$ and $q$ distinct odd primes, suppose $F(z)$ is a Reinhardt polynomial for $n$, and suppose $f_1(z)$ and $f_2(z)$ form a nontrivial decomposition of $F(z)$, and define the sequence $\{s_i\}$ as in \eqref{eqnST}.
Then there exist integers $i$ and $j$ with $0\leq i < j < pq$ such that 
$i \equiv j$ mod $2$,
$s_i = 1$, $s_j = 0$, and there are an even number of nonzero terms between $s_i$ and $s_j$:
$\sum_{k=i+1}^{j-1} |s_k| \equiv 0$ mod $2$.
\end{lem}

\begin{proof}
It suffices to establish the result for $s_i = \pm1$, since if $s_i = -1$, we may use $i\pm q$ and $j\pm q$ in place of $i$ and $j$.
We first note that if a coefficient sequence of the form $s_\ell 0 0$ ever occurs with $s_\ell \neq 0$, then we may select $i = \ell$ and $j = \ell + 2$.

Choose $k$ with $0 \leq k < q$ such that $s_k = 0$.
Such an integer $k$ must exist since $f_1(z)$ is nontrivial.
Then $s_{k+q}=0$ as well.
Consider the sequence $0 s_{k+1} \cdots s_{k+q-1} 0$.
Using the observation above, we may assume that each $0$ in this sequence is isolated.
Since $k$ and $k+q$ have opposite parity, we conclude that there exists a string of consecutive nonzero terms of even length within this sequence, beginning at $s_i$ for some $i$, and ending at $s_{i+2m-1}$ for some positive integer $m$, so that $s_{i+2m}=0$.
Selecting $j=i+2m$ completes the proof.
\end{proof}

We may now prove the proposition.

\begin{proof}[Proof of Proposition~\ref{prop:pqRep}]
Assume that there exists a nontrivial decomposition of $F(z)$,
\[
F(z) = f_1(z)\Phi_p(-z^q) + f_2(z)\Phi_q(-z^p),
\]
and define the sequences $\{s_i\}$ and $\{t_j\}$ as in \eqref{eqnST}.
Using Lemma~\ref{lemParity}, select integers $i$, $j$, $k$, and $\ell$, each in $[0,pq)$, so that 
$i \equiv j$ mod $2$, $k \equiv \ell$ mod $2$, $s_i = t_k = 1$, $s_j = t_\ell = 0$, $i < j$, $k < \ell$, and there are an even number of nonzero terms between $s_i$ and $s_j$, as well as between $t_k$ and $t_\ell$.
Select integers $p'$ and $q'$ so that $p p' + q q' = 1$, and assume without loss of generality that $p'$ is even, so $q q' \equiv 1$ mod $2p$ and $pp' = 1 + q$ mod $2q$.
Let
\[
\hat{F}(z)=(1-z^{pq} + z^{2pq} - z^{3pq})F(z) = 
    \sum_{i=0}^{4pq-1} u_i z^i,
\]
so that the nonzero coefficients of $\hat{F}(z)$ alternate $\pm1$, and let
\[
(1-z^{pq}+z^{2pq}-z^{3pq})f_1(z)\Phi_p(-z^q) = \sum_{i=0}^{4pq-1} s_i z^i
\]
and
\[
(1-z^{pq}+z^{2pq}-z^{3pq})f_2(z)\Phi_q(-z^p) = \sum_{j=0}^{4pq-1} t_j z^j.
\]
Select integers $e_{i,k}$ in $[0, 2pq)$, and $e_{i, \ell}$ and $e_{j,k}$ each in $[2pq,4pq)$, so that
\begin{align*}
e_{i,k} &\equiv i p p' + k q q'  \mod{2pq},\\
e_{i, \ell} &\equiv i p p' + \ell q q'  \mod{2pq},\\
e_{j,k} &\equiv j p p' + k q q' \mod{2pq}.
\end{align*}
We then obtain that
\begin{align*}
e_{i,k} &\equiv i + q(i+k) \mod{2q}, &
e_{i,k} &\equiv k \mod{2p},\\
e_{i,\ell} &\equiv i + q(i+\ell) \mod{2q}, &
e_{i,\ell} &\equiv \ell \mod{2p},\\
e_{j,k} &\equiv j + q(j+k) \mod{2q}, &
e_{j,k} &\equiv k \mod{2p},
\end{align*}
and so
\begin{align*}
s_{e_{i,k}} &= \pm s_i = \pm1, &
t_{e_{i,k}} &= t_k = 1,\\
s_{e_{i,\ell}} &= \pm s_i = \pm1, & 
t_{e_{i,\ell}} &= t_\ell = 0,\\
s_{e_{j,k}} &= \pm s_j = 0, & 
t_{e_{j,k}} &= t_k = 1.
\end{align*}
However, since $s_{e_{i,k}}+t_{e_{i,k}} = u_{e_{i,k}} \in\{-1,0,1\}$ and $t_{e_{i,k}}=1$, we must have $s_{e_{i,k}}=-1$.
If $i$ and $k$ had the same parity, then $e_{i,k}\equiv i$ mod $2q$, and this would imply that $s_{e_{i,k}}=s_i=1$, so in fact $i\not\equiv k$ mod $2$, and therefore $i\not\equiv \ell$ mod $2$ and $j\not\equiv k$ mod $2$ as well.
Thus, $e_{i,k}\equiv e_{i,\ell}\equiv i+q$ mod $2q$ and so $s_{e_{i,\ell}}=-1$.
Thus,
\begin{equation}\label{eqnUvals}
\begin{split}
u_{e_{i,k}} &= s_{e_{i,k}} + t_{e_{i,k}} = 0,\\
u_{e_{i, \ell}} &= s_{e_{i,\ell}} + t_{e_{i,\ell}} = -1,\\
u_{e_{j,k}} &= s_{e_{j,k}} + t_{e_{j,k}} = 1.
\end{split}
\end{equation}
Consider the number of nonzero values of $s_m$ and $t_m$ with $m$ lying strictly between $e_{i,k}$ and $e_{j,k}$.
Notice that $e_{i,k} < e_{j,k}$ and $e_{i,k} < e_{i,\ell}$ by construction.
Since $e_{i,k} \equiv e_{j,k} \equiv k$ mod $2p$, using the fact that $t_j=t_{j+2p}$ we find that there are an odd number of such terms $t_m$, and because $e_{i,k} \equiv i+q$ mod $2q$ and $e_{j,k} \equiv j+q$ mod $2q$, there are an even number of such terms $s_m$, due to the manner of choosing $i$ and $j$.
This implies that there are an odd number of nonzero terms $u_m$ with index lying strictly between $e_{i,k}$ and $e_{j,k}$.
A similar argument can be made using $e_{i,k}$ and $e_{i, \ell}$, so there are an odd number of nonzero terms $u_m$ with index lying strictly between these two values.
It follows then that there are an odd number of nonzero coefficients of $\hat{F}(z)$ lying strictly between $e_{i,\ell}$ and $e_{j,k}$.
However, from \eqref{eqnUvals} these terms have opposite sign, yet the nonzero coefficients of $\hat{F}(z)$ must alternate in sign, so we obtain a contradiction.
Therefore, no nontrivial decomposition exists.
\end{proof}

\section{Proofs of Theorems~\ref{thmSporadic} and~\ref{thmNumber}}\label{secBuildingPolygons}

\subsection{Proof of Theorem~\ref{thmSporadic}}\label{subsecSporadic}
Let $n=pqr$, with $p$ and $q$ distinct odd primes and $r\geq2$.
Using \eqref{eqnCycGen}, it suffices to construct nontrivial polynomials $f_1(z)$ and $f_2(z)$ so that the polynomial
\[
F(z) = f_1(z)\Phi_q(-z^{pr}) + f_2(z)\Phi_p(-z^{qr}).
\]
has the required structure.
(We set $f_0(z)=0$, as well as all of the $f_i(z)$ that correspond to other prime factors of $r$.)
We select $f_1(z)=1-z$, and let $g_1(z)=f_1(z)\Phi_q(-z^{pr})$ for convenience.
Arrange the coefficients of $g_1(z)$ as $q$ rows of length $pr$:
\begin{equation}\label{eqnLayoutG1}
\begin{matrix}
+&-&0&0&\cdots&0\\
-&+&0&0&\cdots&0\\
\vdots&\vdots&\vdots&\vdots&\ddots&\vdots\\
+&-&0&0&\cdots&0,
\end{matrix}
\end{equation}
where we write $+$ for $+1$ and $-$ for $-1$.
Next, we select $f_2(z)$ to have a coefficient sequence of the form
\[
\begin{matrix}
0&A_1&B_1&A_2&B_2&\cdots&A_t&B_t&C,
\end{matrix}
\]
where each $A_i$ and $B_i$ is a sequence of length $r$ over $\{-1,0,1\}$, $C$ is a sequence of length $r-1$ over $\{-1,0,1\}$, and $t=(q-1)/2$, so there are $qr$ coefficients in all.
In addition, we require that each $A_i$, $B_i$, and $C$ must have an odd number of nonzero terms which alternate in sign, and further the first (and last) nonzero term of each $A_i$ and $C$ must be $+1$, while the first (and last) nonzero term of each $B_i$ must be $-1$.

It is straightforward to count the number of different such polynomials $f_2(z)$.
Each possible $A_i$ or $B_i$ corresponds to a subset of $\{1,\ldots,r\}$ with odd length (the subset corresponds to the positions of the nonzero coefficients), and in the same way $C$ corresponds to a subset of $\{1,\ldots,r-1\}$ with odd length.
Since $\sum_{k\geq0}\binom{m}{2k+1}=2^{m-1}$ for any positive integer $m$, it follows that the number of different possible polynomials $f_2(z)$ is $2^{(r-1)(q-1)+(r-2)}=2^{q(r-1)-1}$.

For convenience, we let $g_2(z)=f_2(z)\Phi_p(-z^{qr})$, so that $g_2(z)$ has the coefficient sequence
\begin{equation}\label{eqnLayoutG2}
\begin{matrix}
0&A_1&B_1&A_2&B_2&\cdots&A_t&B_t&C\\
0&\negative{A_1}&\negative{B_1}&\negative{A_2}&\negative{B_2}&\cdots&\negative{A_t}&\negative{B_t}&\negative{C}\\
0&A_1&B_1&A_2&B_2&\cdots&A_t&B_t&C\\
\vdots&\vdots&\vdots&\vdots&\vdots&\ddots&\vdots&\vdots&\vdots\\
0&A_1&B_1&A_2&B_2&\cdots&A_t&B_t&C,
\end{matrix}
\end{equation}
where  there are $p$ rows of length $qr$, and here $\negative{X}$ denotes the sequence created by negating each term of the sequence $X$.

We claim first that every polynomial $F(z)=g_1(z)+g_2(z)$ produced in this way is a Reinhardt polynomial for $n$.
Since it is clear that $F(0)=1$, $\deg(F)<n$, the last nonzero coefficient of $F(z)$ is $+1$, and $\Phi_{2n}(z)\mid F(z)$, we need only verify that the nonzero coefficients of $F(z)$ alternate $\pm1$.
Certainly the polynomial $g_2(z)$ already has this property, so we need only verify that adding $g_1(z)$ to it maintains this pattern.
For convenience, let $\alpha_k$ denote the $k$th block of length $r$ from $g_1(z)$, so that $\alpha_{2kp}=(+1,-1,0,\ldots,0)$ for $0\leq k\leq(q-1)/2$, $\alpha_{(2k+1)p}=(-1,+1,0,\ldots,0)$ for $0\leq k<(q-1)/2$, and all other $\alpha_k$ are entirely $0$.
We consider the effect of adding each block $\alpha_{kp}$ to the coefficient sequence for $g_2(z)$.

Consider first the block $\alpha_0$.
If $A_1$ begins with $+1$, then this coefficient cancels with the $-1$ in the second position of $\alpha_0$, and so the coefficient sequence for $F(z)$ begins with $+1$, then the remaining coefficients of $A_1$ follow, and any nonzero values here begin with $-1$ and alternate in sign.
On the other hand, if $A_1$ begins with $0$, then adding $\alpha_0$ simply adds an additional $(+1,-1)$ pair at the beginning of the sequence.

Next, consider the block $\alpha_{2kp}$ with $1\leq k\leq(q-1)/2$.
The $(+1,-1)$ pair of this block from \eqref{eqnLayoutG1} overlays the coefficient sequence of $g_2(z)$ either in the last position of some $B_i$ and the first position of either $A_{i+1}$ or $C$, or in the last position of some $\negative{A_i}$ and the first position of $\negative{B_i}$.
Suppose the overlay occurs at the last position of $B_i$ and the first position of $A_{i+1}$.
If $B_i$ ends with $-1$ and $A_{i+1}$ begins with $+1$, then this pair is canceled by $\alpha_{2kp}$.
If $B_i$ ends with $-1$ and $A_{i+1}$ begins with $0$, then the addition of $\alpha_{2kp}$ in effect simply moves the $-1$ by one position to the right.
If $B_i$ ends with $0$ and $A_{i+1}$ begins with $+1$, then $\alpha_{2kp}$ moves the $+1$ by one position to the left.
Finally, if $B_i$ ends with $0$ and $A_{i+1}$ begins with $0$, then adding $\alpha_{2kp}$ inserts an additional $(+1,-1)$ pair between a $-1$ and a $+1$, so the alternating sign pattern is maintained.
The argument for the other cases is similar.

Finally, consider the block $\alpha_{(2k+1)p}$ with $0\leq k<(q-1)/2$.
Now the $(-1,+1)$ pair of this block occurs either at the last position of $A_i$ and the first position of $B_i$, or at the last position of $\negative{B_i}$ and the first position of either $\negative{A_{i+1}}$ or $\negative{C}$.
The proof here is similar to the one for the even-indexed blocks.
This completes the proof that $F(z)$ is a Reinhardt polynomial for $n$.

Next, we claim that only $2^{r-2}$ of the polynomials that may be constructed with this method produce periodic Reinhardt polygons, so that the vast majority are in fact sporadic.

Let $u_k$ denote the coefficient of $z^k$ in $F(z)$, for $0\leq k<n$.
The polynomial $F(z)$ corresponds to a periodic Reinhardt polygon if there exists a positive integer $d\mid n$ such that $u_k = -u_{k+d}$ for $0\leq k<n-d$.
In this case, we say that $F(z)$ is \textit{$d$-periodic}.
Let $m=n/d$.
Since the number of nonzero coefficients of $F(z)$ is odd, it follows that the number of nonzero coefficients $u_k$ with $0\leq k<d$ is odd, and so $m$ is odd.
By replacing $d$ with an odd multiple of it if necessary, we may assume that $m$ is prime.
We consider three cases to complete the proof.

First, suppose that $m=p$, so $d=qr$.
Since $F(z)$ and $g_2(z)$ are both $qr$-periodic, it follows that $g_1(z)$ must be $qr$-periodic as well.
Let $\gamma_k$ denote the $k$th block of size $r$ of the coefficients of $g_1(z)$.
By construction, the only nonzero such blocks are $\gamma_{kp}$, for $0\leq k<q$, but by hypothesis $\gamma_0=\negative{\gamma_q}$.
Clearly, $q$ is not a multiple of $p$, so this is a contradiction.

Second, suppose that $m=q$, so $d=pr$.
Since $F(z)$ and $g_1(z)$ are both $pr$-periodic, then so is $F(z)-g_1(z)=g_2(z)$, and thus so is $g_2(z)/z$.
Let $A_{t+1}$ denote the sequence of length $r$ obtained by appending $0$ to $C$, so that $g_2(z)/z$ has the coefficient sequence
\begin{equation}\label{eqnShiftedG2}
\begin{matrix}
A_1&B_1&A_2&B_2&\cdots&A_t&B_t&A_{t+1}\\
\negative{A_1}&\negative{B_1}&\negative{A_2}&\negative{B_2}&\cdots&\negative{A_t}&\negative{B_t}&\negative{A_{t+1}}\\
\vdots&\vdots&\vdots&\vdots&\ddots&\vdots&\vdots&\vdots\\
A_1&B_1&A_2&B_2&\cdots&A_t&B_t&A_{t+1},
\end{matrix}
\end{equation}
organized as $p$ rows of size $qr$.
Let $\beta_k$ denote the $k$th block of size $r$ in this sequence, so $\beta_0=A_1$, $\beta_1=B_1$, etc.
By hypothesis, we have that $\beta_0=\negative{\beta_p}=\beta_{2p}=\negative{\beta_{3p}}=\cdots=\beta_{(q-1)p}$, and we note that each index in this list is unique modulo $q$.
It follows that $A_1$ precisely matches exactly one block in each column of \eqref{eqnShiftedG2}, and by observing the parity of the indices we conclude that $A_1=\cdots=A_{t+1}=\negative{B_1}=\cdots=\negative{B_t}$.
Since $A_{t+1}$ ends with $0$, it follows that if $m=q$ then $f_2(z)$ has coefficients with the form
\[
\begin{matrix}
0&C&0&\negative{C}&0&C&0&\cdots&0&C,
\end{matrix}
\]
where $C$ has length $r-1$ and nonzero coefficients alternating $+1$ and $-1$, with $+1$ for its first and last nonzero coefficient.
Thus, there are exactly $2^{r-2}$ polynomials with $m=q$.

Last, suppose that $m=\ell$, where $\ell$ is an odd prime dividing $r$, with $\ell\not\in\{p,q\}$.
Let $r=\ell s$, so that $d=pqs$.
Group the coefficients of $F(z)$ into blocks of size $qs$, and denote these blocks by $\delta_k$ with $0\leq k<p\ell$.
Arrange these blocks as $p$ rows of size $\ell$:
\begin{equation}\label{eqnDeltas}
\begin{matrix}
\delta_0 & \delta_1 & \cdots & \delta_{\ell-1}\\
\delta_\ell & \delta_{\ell+1} & \cdots & \delta_{2\ell-1}\\
\vdots&\vdots&\ddots&\vdots\\
\delta_{(p-1)\ell} & \delta_{(p-1)\ell+1} & \cdots & \delta_{p\ell-1},
\end{matrix}
\end{equation}
so that the coefficients are arranged into $p$ rows of size $qr$, just as in \eqref{eqnLayoutG2}.
Thus, the first integer in each row of \eqref{eqnDeltas} corresponds to the coefficient positions $kqr$, with $0\leq k<p$.
When $g_1(z)$ is added to $g_2(z)$, from \eqref{eqnLayoutG1} we see that the only coefficients affected occur in pairs beginning at positions which are multiples of $pr$.
It follows that $\delta_0$ begins with $+1$, but $\delta_{k\ell}$ begins with $0$ for $1\leq k<p$.
Using the hypothesis of periodicity with $m=\ell$, we then find that the first integer in $\delta_{2kp}$ is $+1$ for $0\leq k\leq(\ell-1)/2$, and the first value in $\delta_{(2k+1)p}$ is $-1$ for $0\leq k\leq(\ell-3)/2$.
In particular, $\delta_p$ must begin with $-1$, and $\delta_{p+\ell}$ must begin with $0$ (since $\delta_\ell$ begins with $0$), and so either $\delta_p$ or $\delta_{p+\ell}$ must have been altered by one of the nonzero blocks of $g_1(z)$.
Suppose $\delta_p$ was altered.
Then the first position of $\delta_p$ is either $(2k+1)pr$ or $2kpr+1$ for some $k$.
The first case implies that $q=(2k+1)\ell$, and the second produces $ps(q-2k\ell)=1$, and both of these are clearly impossible since $p$ and $q$ are prime.
Suppose then that $\delta_{p+\ell}$ was altered.
Since $0$ occurs in the first position here, we have that $(p+\ell)qs$ must be either $kpr$ or $kpr+1$ for some $k$.
The first case implies that $p\mid\ell q$, which is impossible, and the second case yields $s(pq+\ell q -k\ell p) = 1$, so $s=1$ and thus $\ell=r$.
In particular, we have that $pq\equiv1$ mod $r$ in this case.

We can eliminate this last possibility by considering $\delta_{2p}$ and $\delta_{2p+r}$.
Since $\delta_{2p}$ begins with $+1$ and $\delta_{2p+r}$ starts with $0$, one of these blocks must have been altered by a nonzero block of $g_1(z)$.
If $\delta_{2p}$ was altered, then since $s=1$ we have that $2pq$ must equal either $2kpr$ or $(2k+1)pr+1$ for some $k$, and it is straightforward to show that neither of these is possible.
If $\delta_{2p+r}$ was changed, then $(2p+r)q$ must equal either $kpr$ or $kpr+1$ for some $k$.
The former possibility is easily dismissed; the latter produces $2pq + qr - kpr = 1$, and so $2pq\equiv1$ mod $r$.
However, from the analysis of $\delta_p$, we know that $pq\equiv1$ mod $r$, and so we conclude $r=1$, a contradiction.
Thus, $F(z)$ cannot be periodic with $m=\ell$.
This completes the proof of Theorem~\ref{thmSporadic}.
\qed

\subsection{Proof of Theorem~\ref{thmNumber}}\label{subsecNumber}
Suppose $p$ and $q$ are fixed odd primes with $p<q$, let $\epsilon>0$, and suppose that $r$ is a positive integer having no odd prime divisor less than $p$.
We first generalize the construction of the proof of Theorem~\ref{thmSporadic}, by allowing freedom in the construction of $f_1(z)$.
Select a nontrivial, proper subset $S$ of $\{0,1,\ldots,p-1\}$, and let
\[
f_1(z) = \sum_{s\in S} (-1)^s z^{rs} (1-z),
\]
so that the case $S=\{0\}$ corresponds to the polynomial $f_1(z)=1-z$ employed in the prior proof.
Let $g_1(z)=f_1(z)\Phi_q(-z^{pr})$ as before, and construct $f_2(x)$ and $g_2(z)$ as in the previous proof, by selecting qualifying sequences $A_i$, $B_i$, and $C$, for $1\leq i\leq t=(q-1)/2$.
Let $F(z) = g_1(z) + g_2(z)$.

Since the nonzero coefficients of $g_1(z)$ overlap $g_2(z)$ at the boundaries of the blocks $A_i$, $B_i$, and $C$, one may verify in the same way as the prior proof that the coefficients of $F(z)$ are all $-1$, $0$, and $1$, with the nonzero coefficients alternating in sign.
Also, it is easy to see that different choices for $g_1(z)$ and $g_2(z)$ can never produce the same polynomial $F(z)$: if $g_1(z)+g_2(z) = g_1^*(z) + g_2^*(z)$, then $g_1(z)-g_1^*(z)=g_2^*(z)-g_2(z)$, and the left side is $pr$-periodic and the right side is $qr$-periodic, so both sides must be $r$-periodic, and $g_2(0)-g_2^*(0)=0$ then implies that $g_1(z)-g_1^*(z)=0$.
It follows that the total number of different polynomials that can be produced by using this construction is $(2^p-2)2^{q(r-1)-1}$.

We may also determine the number of polynomials $F(z)$ arising from this construction that exhibit a periodic structure.
Suppose that $F(z)$ is $d$-periodic, and that $m=n/d$ is prime.
If $m=p$, then as before $g_1(z)$ is $qr$-periodic, and it follows that $S=\{\}$ or $S=\{0,\ldots,p-1\}$, but these choices were disallowed.
If $m=q$, then as in the prior proof we find that $A_1=\cdots=A_t=C0=\negative{B_1}=\cdots=\negative{B_t}$, so that there are $(2^p-2)2^{r-2}$ such polynomials.
Finally, an argument similar to that employed in the previous proof shows that no polynomials $F(z)$ have $m=\ell$ with $\ell$ an odd prime divisor of $r$ and $\ell\not\in\{p,q\}$.

Not all polynomials $F(z)$ constructed by using this method are Reinhardt polynomials, since we do not guarantee that $F(0)=1$.
(For this, one must require that $0\in S$.)
However, each such $F(z)$ is equivalent to a Reinhardt polynomial under a dihedral action, and we may therefore determine a lower bound on the number of different sporadic Reinhardt polygons with $n=pqr$ sides.
By accounting for equivalence classes, we conclude that this method constructs at least
$(2^p-2)\left(\frac{2^{q(r-1)-1}}{2pqr} - 2^{r-2}\right)$ different sporadic Reinhardt polygons.

Since $p$ is the smallest odd prime divisor of $n$, from \cite[Cor.~2]{Mossinghoff11} we have that
\[
E_0(pqr) \sim \frac{2^{qr}}{4qr}
\]
as $r$ grows large.
Therefore,
\[
\frac{E_0(pqr)}{E_1(pqr)} \leq \frac{\frac{1}{4qr}\cdot 2^{qr}(1 + o(1))}{\left(2^p-2\right)\left(\frac{2^{q(r-1)-1}}{2pqr} - 2^{r-2}\right)} \leq \frac{p2^q}{2^p-2} + o(1)
\]
for qualifying integers $r\to\infty$, and so
\[
\frac{E_1(pqr)}{E(pqr)} = \frac{1}{\frac{E_0(pqr)}{E_1(pqr)}+1} > \frac{2^p - 2}{p2^q + 2^p - 2} - \epsilon
\]
for sufficiently large such $r$.
\qed

Of course, in practice we may build additional Reinhardt polygons for $n$ by selecting $p$ and $q$ in other ways in our construction.
However, it is possible that some of the same polygons (up to dihedral equivalence) will be constructed for different choices of $p$ and $q$, so in general an inclusion/exclusion argument would need to be employed to improve the lower bound of Theorem~\ref{thmNumber} by using this construction.

\section{Constructing sporadic Reinhardt polygons}\label{secConstructions}

In \cite{Mossinghoff11}, the exact value of $E_1(n)$ was computed for $24$ different values of $n$ where $E_1(n)>0$.
In addition, a partial count for $n=105$ was reported, and some evidence was presented that $n=105$ may be the smallest positive integer where $E_1(n)>E_0(n)$.
It is natural then to determine the number of different Reinhardt polygons that can be constructed by using the method of Section~\ref{secBuildingPolygons} for these $25$ values of $n$.
We report here on some computations made to investigate this.

\begin{table}[t]
\caption{Number of sporadic Reinhardt polygons constructed.}\label{tableExactCounts}
\begin{tabular}{rccrr}
\multicolumn{1}{c}{$n$} & Factorization & $r$ & $E_1(n)$ & $\constr{n}$\\\hline
$30$ & $2\cdot3\cdot5$ & $2$ & $3$ & $3$\\
$42$ & $2\cdot3\cdot7$ & $2$ & $9$ & $9$\\
$45$ & $3^2\cdot5$ & $3$ & $144$ & $144$\\
$60$ & $2^2\cdot3\cdot5$ & $4$ & $4\,392$ & $3\,492$\\
$63$ & $3^2\cdot7$ & $3$ & $1\,308$ & $1\,308$\\
$66$ & $2\cdot3\cdot11$ & $2$ & $93$ & $93$\\
$70$ & $2\cdot5\cdot7$ & $2$ & $27$ & $27$\\
$75$ & $3\cdot5^2$ & $5$ & $153\,660$ & $107\,400$\\
$78$ & $2\cdot3\cdot13$ & $2$ & $315$ & $315$\\
$84$ & $2^2\cdot3\cdot7$ & $4$ & $161\,028$ & $150\,444$\\
$90$ & $2\cdot3^2\cdot5$ & $6$ & $5\,385\,768$ & $3\,371\,568$\\
$99$ & $3^2\cdot11$ & $3$ & $192\,324$ & $192\,324$\\
$102$ & $2\cdot3\cdot17$ & $2$ & $3\,855$ & $3\,855$\\
$110$ & $2\cdot5\cdot11$ & $2$ & $279$ & $279$\\
$114$ & $2\cdot3\cdot19$ & $2$ & $13\,797$ & $13\,797$\\
$117$ & $3^2\cdot13$ & $3$ & $2\,587\,284$ & $2\,587\,284$\\
$130$ & $2\cdot5\cdot13$ & $2$ & $945$ & $945$\\
$140$ & $2^2\cdot5\cdot7$ & $4$ & $633\,528$ & $478\,548$\\
$154$ & $2\cdot7\cdot11$ & $2$ & $837$ & $837$\\
$170$ & $2\cdot5\cdot17$ & $2$ & $11\,565$ & $11\,565$\\
$182$ & $2\cdot7\cdot13$ & $2$ & $2\,835$ & $2\,835$\\
$190$ & $2\cdot5\cdot19$ & $2$ & $41\,391$ & $41\,391$\\
$238$ & $2\cdot7\cdot17$ & $2$ & $34\,695$ & $34\,695$\\
$286$ & $2\cdot11\cdot13$ & $2$ & $29\,295$ & $29\,295$
\end{tabular}
\end{table}

Table~\ref{tableExactCounts} displays the $24$ integers of the form $n=pqr$, with $p$ and $q$ distinct odd primes and $r\geq2$ for which $n-\varphi(2n)\leq46$.
(The bound of $46$ was selected due to computational constraints.)
The fourth column of this table shows the exact value of $E_1(n)$ from \cite{Mossinghoff11}, and the last column exhibits $\constr{n}$, the number of different sporadic Reinhardt polygons that can be constructed by using the method of the proof of Theorem~\ref{thmNumber}, by selecting values for $p$, $q$, and $r$, and then checking all possible nontrivial proper subsets $S$, all permissible sequences $A_i$ and $B_i$ for $1\leq i\leq(q-1)/2$, and all allowable sequences $C$.
In each case, the value of $r$ is forced, and $p$ and $q$ may be selected in two different ways; both ways were checked in computing $\constr{n}$.
As in \cite{Mossinghoff11}, we count two polygons to be distinct only if one cannot be obtained from the other by some combination of rotations and flips.

We see that our construction produces all of the sporadic Reinhardt polygons for the nineteen values of $n$ in the table where $r=2$ or $r=3$.
For example, when $n=30$, select $p=5$, $q=3$, $r=2$, $A_1=0+$, $B_1=0-$, and $C=+$.
If $A_2=+0$ and $B_2=0-$, then we obtain the polygon of Figure~\ref{figSporadic30}(a); if $A_2=0+$ and $B_2=0-$, then we construct the polygon of Figure~\ref{figSporadic30}(b); finally, if $A_2=0+$ (or $+0$) and $B_2=-0$, then we create the polygon of Figure~\ref{figSporadic30}(c).
(Here, we have normalized each dihedral composition of $30$ so that the largest part occurs first.)
This accounts for all three sporadic Reinhardt triacontagons.
For the values of $n$ where $r>3$ ($r=4$ for $n\in\{60,84,140\}$, $r=5$ for $n=75$, and $r=6$ for $n=90$), our method constructs a substantial proportion, but not all, of the sporadic Reinhardt polygons with $n$ sides.

For $n=105$, in \cite{Mossinghoff11} it was found that the number of periodic Reinhardt polygons is $E_0(105)=245\,518\,324$, and some evidence was presented that $n=105$ may be the smallest integer where $E_1(n)>E_0(n)$.
Table~\ref{tableSporadic105} displays $E_1(105,m)$, the total number of sporadic Reinhardt $105$-gons whose corresponding dihedral composition has largest part $m$, for $m=2$ and $m\geq12$, as computed in \cite{Mossinghoff11}.
This table also exhibits the value of $\constr{105,m}$, the number of sporadic Reinhardt $105$-gons with largest part $m$ that may be constructed by using the method of the proof of Theorem~\ref{thmNumber}.
For this calculation, we considered only sets $S$ that contained $0$ when constructing $f_1(z)$, since these in fact sufficed for the results of Table~\ref{tableExactCounts}.
All six different possible choices of $p$, $q$, and $r$ were considered.
In all, our method constructs $126\,714\,582$ different sporadic Reinhardt $105$-gons, including $3\,492\,473$ of the $12\,978\,294$ polygons with $m=2$ or $m\geq12$, or about $27\%$ of this portion.
By using the values we computed for $\constr{105,m}$ for $3\leq m\leq11$, we might expect then that $E_1(105)$ is close to $470$ million, or nearly twice the value of $E_0(105)$.
This then provides some additional empirical evidence that sporadic Reinhardt polygons first outnumber the periodic ones at $n=105$.

\begin{table}[t]
\caption{Sporadic Reinhardt $105$-gons with largest part $m$.}\label{tableSporadic105}
\begin{tabular}{crr@{\quad}|@{\quad}crr}\\
$m$ & $E_1(105,m)$ & $\constr{105,m}$ & $m$ & $E_1(105,m)$ & $\constr{105,m}$\\\hline
$2$ & $1\,831$ & $378$ & $15$ & $1\,227\,719$ & $260\,920$\\
$3$ & ? & $869\,572$ & $16$ & $544\,966$ & $132\,839$\\
$4$ & ? & $12\,319\,890$ & $17$ & $250\,440$ & $66\,113$\\
$5$ & ? & $27\,537\,337$ & $18$ & $117\,075$ & $32\,391$\\
$6$ & ? & $32\,613\,532$ & $19$ & $55\,382$ & $16\,362$\\
$7$ & ? & $19\,788\,045$ & $20$ & $20\,234$ & $6\,145$\\
$8$ & ? & $13\,529\,809$ & $21$ & $16\,580$ & $4\,612$\\
$9$ & ? & $8\,758\,704$ &$22$ & $5\,609$ & $2\,044$\\
$10$ & ? & $4\,936\,396$ & $23$ & $2\,144$ & $903$\\
$11$ & ? & $2\,868\,824$ & $24$ & $788$ & $384$\\
$12$ & $5\,749\,059$ & $1\,601\,785$ & $25$ & $242$ & $164$\\
$13$ & $3\,155\,368$ & $941\,576$ & $26$ & $80$ & $64$\\
$14$ & $1\,830\,741$ & $425\,757$ & $27$ & $36$ & $36$
\end{tabular}
\end{table}


\def\lfhook#1{\setbox0=\hbox{#1}{\ooalign{\hidewidth
  \lower1.5ex\hbox{'}\hidewidth\crcr\unhbox0}}}
\providecommand{\bysame}{\leavevmode\hbox to3em{\hrulefill}\thinspace}
\providecommand{\MR}{\relax\ifhmode\unskip\space\fi MR }
\providecommand{\MRhref}[2]{%
  \href{http://www.ams.org/mathscinet-getitem?mr=#1}{#2}
}
\providecommand{\href}[2]{#2}

\end{document}